\documentclass[11pt]{article}
\usepackage{graphicx}
\usepackage{diagbox}
\usepackage[normalem]{ulem}
\usepackage[margin=1truein]{geometry}
\usepackage{setspace}
\usepackage{amsmath,amsfonts,amssymb,amsthm,mathtools}  
\usepackage{enumerate}   
\usepackage[ruled,vlined,boxruled]{algorithm2e}
\usepackage{algorithmicx}
\usepackage[font=small,labelfont=bf]{caption}
\SetKwInOut{Input}{Input}
\SetKwInOut{Output}{Output}
\usepackage{autobreak}
\usepackage{float}
\usepackage{indentfirst}
\usepackage{etoolbox}
\usepackage{url}
\usepackage[colorlinks,linkcolor=red,citecolor={blue}]{hyperref}
\usepackage[all]{hypcap}
\hypersetup{hypertexnames=false}
\usepackage[square,numbers]{natbib}
\def\X{\mathbb{X}}
\def\A{\mathbb{A}}
\def\R{\mathbb{R}}
\def\N{\mathbb{N}}

\def\F{\mathbb{F}}

\newtheorem{thm}{Theorem}[section]

\newtheorem{lem}[thm]{Lemma}
\newtheorem{cor}[thm]{Corollary}
\newtheorem{rem}[thm]{Remark}

\DeclarePairedDelimiter\abs{\lvert}{\rvert}
\DeclarePairedDelimiterX{\norm}[1]{\lVert}{\rVert}{#1}
\DeclarePairedDelimiterX{\normo}[1]{\lVert}{\rVert_*}{#1}
\makeatletter

\numberwithin{equation}{section}
\allowdisplaybreaks[4]

\begin{document}
	\title{Differential Properties of Objective and Value Functions for Discounted Markov Decision Processes with Finite State and Action Sets}
    \author{Eugene A. Feinberg\thanks{Corresponding author, Department of Applied Mathematics and
Statistics, Stony Brook University, Stony Brook, NY 11794-3600, USA, eugene.feinberg@stonybrook.edu, https://orcid.org/0000-0002-8263-0772}       
    \and    
    Gaojin He\thanks{School of Engineering and Sciences,
        Tecnológico de Monterrey, Mexico, gaojin.he@tec.mx}}
	\date{\today}
	\maketitle
	\begin{abstract}
		This paper studies differential properties of objective functions for discrete-time discounted Markov decision processes (MDPs) with finite state and action sets, when differentiation is considered with respect to the discount factor. Given a fixed discount factor for an MDP, we estimate the norms and seminorms of the objective functions, their derivatives, and the value function at a different discount factor. We then use these estimates to prove that objective functions, their derivatives, and value functions are Lipschitz continuous in discount factors. As a result, if an MDP is solved with a fixed discount factor, then we can estimate the value function at a different discount factor without solving it again.\\ 
		
		\noindent
	\end{abstract}

\section{Introduction}
   
   This paper deals with infinite-horizon discounted Markov decision processes (MDPs) with finite state and action sets. We investigate the properties of objective functions, derivatives of objective functions, and value functions of finite MDPs in terms of discount factors, and use them to derive differential results of objective functions and value functions regarding monotonicity, estimation, and Lipschitz continuity. As a result, if an MDP has already been solved for a fixed discount factor, then these results provide fast ways to estimate the value function at a different discount factor without solving the MDP again. 
   
   Section~\ref{secp} deals with preliminary results that are used in the proofs of the main results in Sections~\ref{seces} and \ref{seclip}. Section~\ref{seces} provides estimations of objection functions and value functions in discount factors. Theorem~\ref{diffb} provides bounds of higher-order derivatives of objective functions in their lower-order derivatives. Theorems~\ref{mon} and \ref{Vine} shows the monotonicity of some quantities related to the norms and the seminorms of objective functions, their derivatives, and value functions. As a result, Corollaries~\ref{vb}, \ref{ratiob}, \ref{valuein}, \ref{V0} provide different kinds of estimates of the norms and the seminorms of objective functions, their derivatives, and value functions in discount factors. Section~\ref{seclip} provides results regarding Lipschitz continuity of objective functions and value functions in discount factors. Theorem~\ref{olip} provides Lipschitz continuity of objective functions and their derivatives. Theorem~\ref{Vlip} provides Lipschitz continuity of value functions, which strengthens the result in \cite[Lemma 4.2]{TPa} and \cite[Theorem 2]{NIPS}. 
   
   Some secondary results are also worth mentioning. Corollary~\ref{diff0} shows that, for a deterministic policy, if  the norm or the seminorm of the objective function at one discount factor is $0,$ then this norm or seminorm is $0$ for all discount factors. Corollary~\ref{v0} shows that this property also applies to value functions. Corollary~\ref{Value0}(b) implies that, for an irregular discount factor at which the set of optimal policies changes, the value function vector cannot have all entries being identical. Some of these results appeared in \cite{OE} without proofs, and they imply the main results on the optimality extension in \cite{OE}. 
   
   We also generalized some previously published results. It turns out that \cite[Proposition 2.2]{TPa}, \cite[Proposition 6.6.1]{Put}, and \cite[Theorem 2.1]{TPa} are the special cases of zero discount factor of Lemmas~\ref{la}, \ref{lemQ}, and Corollary~\ref{lemeq1} in this paper, respectively.

\section{Definitions and Preliminary Results}\label{secp}
Let $\N,$ $\N^+,$  $\R,$ and $\R^{i\times j}$ be the sets of natural numbers including $0$, positive integers,  real numbers, and matrices with $i$ rows and $j$ columns, respectively, where $i,j\in \N^+.$ We also denote by $\R^j:=\R^{1\times j}$ the set of $j$-dimensional column vectors.

We consider a Markov decision process (MDP) with a finite state space $\X$ consisting of $m$ states $\X:=\{x^i\}_{i=1}^m$ and a finite action space $\A:=\bigcup_{x\in\X}A_x$, where $A_x\ne\emptyset,$ and $A_x$ is the set of available actions at each state $x\in\X$. Each action set $A_x$ consists of $q_x$ actions. Thus, the total number of actions is $q:=\sum_{x=1}^m q_x,$ and this number can be interpreted as the total number of state-action pairs. For each $x\in\X$ and $a\in A_x$ there is a \textbf{one-step reward} $r(x,a)$. Starting from some initial state $x\in\X$, a decision-maker selects an action and collects the corresponding reward, and the state is moved to the next state $y\in\X$ with a probability distribution $p(y|x,a).$ This process continues over an infinite planning horizon. A \textbf{decision rule} is a mapping $\phi:\X\to\A$ such that $\phi(x)\in A_x$ for each $x\in\X$, and a \textbf{Markov policy} $\pi$ is a sequence of decision rules $(\phi_0,\phi_1,\ldots)$, which means the decision maker applies $\phi_n$ as the decision rule at the $n$-th step. There are more general decision rules and policies, but for finite-state MDPs with expected total rewards, it is sufficient to consider the policies in this form; see e.g., \cite[p.~154]{Put}. Therefore, in this paper we only consider Markov policies. A policy $\pi$ is \textbf{deterministic} if $\pi=(\phi,\phi,\ldots).$ Sometimes such policies are called stationary or nonrandomized stationary. We denote such policies by $\phi$ and identify a decision rule $\phi$ with a deterministic policy $\phi.$ Throughout this paper, we do not distinguish between decision rules and deterministic policies. The term ``decision rule" is used when ``one-step policy" is emphasized, while the term ``deterministic policy" refers to the policy that applies the same one-step policy at all steps. The objective is to find a policy maximizing the expected discounted total reward. We denote the set of all decision rules by $\F,$ which is also the set of deterministic policies. Let $\alpha\in[0,1)$ be a \textbf{discount factor}. The total $\alpha$-discounted expected rewards for an MDP under the initial state $x$ and the policy $\phi,$ which is also called the \textbf{objective function} of the policy $\phi,$ is
    \begin{equation}\label{eqexpectation}
        \begin{aligned}
			v_\alpha^{\phi}(x) & :=\mathbb{E}_x^\phi\left[\sum_{t=0}^{\infty}\alpha^tr(x_t,\phi(x_t))\right],
		\end{aligned}
    \end{equation}
where $x_0,\phi_0(x_0),x_1,\phi_1(x_1),\ldots,x_{t-1},\phi_{t-1}(x_{t-1}),x_t,\ldots$ is a trajectory of the process, $t$ is the integer time parameter, and $\mathbb{E}_x^\phi$ is the expectation defined by the initial state $x$ and under the policy $\phi$. A policy $\phi$ is \textbf{optimal} for an infinite-horizon problem, if $v_\alpha^\phi(x)=\sup_{\phi\in\F}v_\alpha^\phi(x)$ for all $x\in\X$. It is well-known that an optimal policy exists; see e.g., \cite[p.~154]{Put}. The \textbf{value function} is defined as $V_\alpha(x):=\max_{\phi\in\F}v_\alpha^\phi(x),\ x\in\X.$ Denote by $v_\alpha^\phi$ and $V_\alpha$ the corresponding column vectors in $\R^m.$ For each $\phi\in\F$, let $P_{i,j}(\phi):=p(x^j|x^i,\phi(x^i))$ for $x^i,x^j\in\mathbb{X}$ and $r(\phi)(x):=r(x,\phi(x))$ for $x\in\mathbb{X}$. Denote by $P(\phi):=\left[P_{i,j}(\phi)\right]_{1\le i,j\le m}\in\R^{m\times m}$ the stochastic matrix and by $r(\phi)$ the column vector of one-step rewards under the decision rule $\phi$. 
For $\alpha\in[0,1)$ and $x\in\X,$ let 
\[D(\alpha):=\left\{\phi\in\F:V_\alpha=v_\alpha^\phi\right\},\qquad A_x(\alpha):=\left\{a\in A_x:a=\phi(x)\ \text{for some $\phi\in D(\alpha)$}\right\}.\]
We view $D(\alpha)$ and $A_x(\alpha)$ as set-valued functions in $\alpha$. For $\alpha\in[0,1),$ $D(\alpha)$ is the set of deterministic optimal policies, and $A_x(\alpha)$ is the set of \textbf{optimal actions} at state $x.$ It is well-known that $D(\alpha)=\{\phi\in\F:\phi(x)\in A_x(\alpha)\  \forall x\in\X\}.$ For $E\subset[0,1),$ let $D(E):=\bigcup_{\alpha\in E}D(\alpha).$ If $E\cap[0,1)=\emptyset,$ then we set $D(E)=\emptyset.$ For $\alpha\in[0,1),$ let
\[D(\alpha-) := \bigcap_{i=1}^\infty D\left(\left(\alpha-\frac{1}{i},\alpha\right)\bigcap\left[0,1\right)\right),\qquad D(\alpha+) := \bigcap_{i=1}^\infty D\left(\left(\alpha,\alpha+\frac{1}{i}\right)\bigcap\left[0,1\right)\right).\]
By this definition, $D(0-)=\emptyset.$ We note that $D(\alpha-),$ $D(\alpha+)$ are called the set of left optimal policies and the set of right optimal policies, respectively. 
We say $\alpha\in(0,1)$ is a \textbf{regular point} if $D(\alpha-)=D(\alpha)=D(\alpha+),$ and the point $\alpha=0$ is regular if $D(0)=D(0+).$ Otherwise, we say $\alpha\in[0,1)$ is an \textbf{irregular point}. There are different types of irregular points; see \cite{TPa} for details. The following lemma is a well-known fact. It states that the interval $[0,1)$ can be partitioned into finitely many subintervals, such that for all discount factors on each of these subintervals, the set of optimal policies remains unchanged, and the value function is analytic in discount factors. See also \cite[Theorem 1.16]{HB}, \cite[Proposition 4.5.3]{sennott1999stochastic}, \cite[Remark 1]{Lewis}, and the proof of \cite[Theorem 5]{BW}.

\begin{lem}\label{lempartition}(a) \textnormal{(\cite[Theorem 3.6(c)]{TPa})} For each $x\in\X$ and $\phi\in\F$, the function $v_\alpha^\phi(x)$ is rational  in $\alpha,$ and the function $V_\alpha(x)$ is continuous and piecewise rational in $\alpha$. There is a unique finite set of points $\{a_i\}_{i=0}^l\subset[0,1]$ such that $a_0=0,$  $a_l=1,$  $a_i<a_{i+1}$ for $0\le i\le l-1,$ the function $D(\alpha)$ is constant on each open interval $\mathcal{I}_i:=(a_i,a_{i+1})$ for $0\le i\le l-1,$ but it has jumps at the points $\{a_i\}_{i=1}^l$ and is not constant on each of half-open intervals $(a_i,a_{i+1}],$ $[a_{i+1},a_{i+2})$ for $0\le i\le l-2.$ In addition, $\{a_1,\cdots,a_{l-1}\}$ is the set of irregular points, and  $V_\alpha=v_\alpha^{\phi^{(i)}}$ for all $\phi^{(i)}\in D(\mathcal{I}_i),\alpha\in\mathcal{I}_i,$ and $0\le i\le l-1.$\medskip\\
(b) \textnormal{(\cite[Proposition 3.10(c) and Remark 3.11]{TPa})} The set-valued function $D(\alpha)$
is upper hemicontinuous in $\alpha\in [0,1),$ and therefore, for each $x\in\X,$  the set-valued functions $A_x(\alpha)$ 
are upper hemicontinuous in $\alpha\in [0,1).$ 
\end{lem}


Let $\pmb{0},\pmb{1}\in\R^m$ be the zero vector and the vector with all entries being $1,$ respectively. Note that $P(\phi)\pmb{1}=\pmb{1}$ for all $\phi\in\F$ and for all $\alpha\in[0,1).$ The following lemma is a result from linear algebra. 

\begin{lem}\label{lem1}
    Let $\phi\in\F,$ $k\in\N,$ and $c\in\R.$ Then
    \begin{enumerate}[(a)]
        \item $P^k(\phi)r(\phi)=\mathbf{0}$ iff $P^t(\phi)r(\phi)=\mathbf{0}$ for all $t\ge\min\left\{k,m-1\right\};$
        \item $P^k(\phi)r(\phi)=c\mathbf{1}$ iff $P^t(\phi)r(\phi)=c\mathbf{1}$ for all $t\ge\min\left\{k,m-1\right\}.$
    \end{enumerate}
\end{lem}
\begin{proof}
    See Appendix~\ref{ap1}.
\end{proof}

For $\phi\in\F$ and $\alpha\in[0,1),$ let $Q_\alpha(\phi)\in\R^{m\times m}$ such that $Q_\alpha(\phi):=[I-\alpha P(\phi)]^{-1}P(\phi)=P(\phi)[I-\alpha P(\phi)]^{-1}=\sum_{k=1}^\infty \alpha^{k-1}P^k(\phi).$ Clearly $Q_0(\phi)=P(\phi).$ Note that $Q^k_\alpha(\phi)\pmb{1}=(1-\alpha)^{-k}\pmb{1}$ for all $k\ge 0,$ $\alpha\in[0,1),$ and $\phi\in\F.$ Let $P^0=I$ for $P\in\mathbb{R}^{m\times m}.$ Then $P^0\mathbf{v}=\mathbf{v}.$ For $\alpha\in[0,1),$ $\phi\in\F,$ and $\mathbf{v}\in\mathbb{R}^m,$ let $G_\alpha(\phi,\mathbf{v}):=1,$ if $\mathbf{v}=c\mathbf{1}$ for some $c\in\mathbb{R},$ and 
\begin{equation}\label{defG}
    G_\alpha(\phi,\mathbf{v})=\max\left\{k\in\N^+: \left\{\pmb{1},Q_\alpha^0(\phi)\mathbf{v},Q_\alpha^1(\phi)\mathbf{v},\cdots,Q_\alpha(\phi)^{k-2}\mathbf{v}\right\}\ \text{are linearly independent}\right\},
\end{equation}
if $\mathbf{v}\ne c\mathbf{1}$ for all $c\in\R.$ Clearly $2\le G_\alpha(\phi,\mathbf{v})\le m$ iff $\mathbf{v}\ne c\mathbf{1}$ for all $c\in\R.$ In particular, $G_\alpha(\phi,\mathbf{v})=1$ if $m=1.$ Thus, $1\le G_\alpha(\phi,\mathbf{v})\le m.$ The following lemma shows that the definition of $G_\alpha$ does not depend on $\alpha.$ 

\begin{lem}\label{G}
    $G_\alpha(\phi,\mathbf{v})=G_0(\phi,\mathbf{v})$ for all $\alpha\in[0,1),$ $\phi\in\F,$ and $\mathbf{v}\in\R^m.$
\end{lem}
\begin{proof}
It is sufficient to show that, for each fixed $\alpha\in[0,1)$ and each $k\ge2,$ the two $m\times k$ matrices
\begin{equation*}
    B_\alpha=\begin{bmatrix}
        \mathbf{1} & \mathbf{v} & Q_\alpha(\phi)\mathbf{v} & \cdots & Q_\alpha^{k-2}(\phi)\mathbf{v}
    \end{bmatrix},\qquad B_0=\begin{bmatrix}
        \mathbf{1} & \mathbf{v} & P(\phi)\mathbf{v} & \cdots & P^{k-2}(\phi)\mathbf{v}
    \end{bmatrix}
    \end{equation*}
    have the same rank. Note that
    \begin{align}\label{123}
    &[I-\alpha P(\phi)]^{k-2}B_\alpha=\begin{bmatrix}
       (1-\alpha)^{k-2}\mathbf{1} & [I-\alpha P(\phi)]^{k-2}\mathbf{v} & [I-\alpha P(\phi)]^{k-3}P(\phi)\mathbf{v} & \cdots & P^{k-2}(\phi)\mathbf{v}
    \end{bmatrix}.
    \end{align}
    Observe that, for the $(i+2)$-th column of the right-hand side of \eqref{123}, where $i=0,1,\cdots,k-2,$
    \begin{align*}
    &\quad[I-\alpha P(\phi)]^{k-2-i}P^i(\phi)\mathbf{v}=P^i(\phi)\mathbf{v}+c_{i,1}P^{i+1}(\phi)\mathbf{v}+c_{i,2}P^{i+2}(\phi)\mathbf{v}+\cdots+c_{i,k-2-i}P^{k-2}(\phi)\mathbf{v}\\
    &=\begin{bmatrix}
        \mathbf{1} & \mathbf{v} & P(\phi)\mathbf{v} & \cdots & P^{k-2}(\phi)\mathbf{v}
    \end{bmatrix}\left[
        \underbrace{0 \ \ \ \cdots \ \ \ 0}_{(i+1)\ \text{times}} \ \ \ 1 \ \ \ c_{i,1} \ \ \ c_{i,2} \ \ \ \cdots \ \ \ c_{i,k-2-i}
    \right]^\mathsf{T},
    \end{align*}
    where $c_{i,j}=\binom{j}{k-2-i}(-\alpha)^j$ for $j=1,2,\cdots,k-2-i,$ and the first equality above follows from the binomial formula. 
    Therefore, \eqref{123} becomes
    \begin{align*}
    [I-\alpha P(\phi)]^{k-2}B_\alpha&=\begin{bmatrix}
        \mathbf{1} & \mathbf{v} & P(\phi)\mathbf{v} & \cdots & P^{k-2}(\phi)\mathbf{v}
    \end{bmatrix}\begin{bmatrix}
              (1-\alpha)^{k-2} & & & &\\
                 & 1 & &  &  \\
                 & \times & 1 & &  \\
                 & \vdots & \vdots & \ddots\\
                 & \times & \times & \cdots & 1 \end{bmatrix}_{k\times k}\\
                 &=B_0\begin{bmatrix}
              (1-\alpha)^{k-2} & & & &\\
                 & 1 & &  &  \\
                 & \times & 1 & &  \\
                 & \vdots & \vdots & \ddots\\
                 & \times & \times & \cdots & 1 \end{bmatrix}_{k\times k},
 \end{align*}
    where all the blank entries are $0,$ and each $``\times"$ represents some real number. Therefore, since both $I-\alpha P(\phi)$ and the last matrix multiplied by $B$ are invertible, we conclude that $B_\alpha$ and $B$ have the same rank.\qedhere
    
\end{proof}

The following lemma is a generalization of \cite[Theorem 2.1]{TPa} which is a special case $\alpha=0.$


\begin{lem}\label{la}
    Let $\alpha\in[0,1),$ $\phi_1,\phi_2\in\F,$ $\mathbf{v}_1,\mathbf{v}_2\in\R^m,$ and $n\in\N.$ Then the following statements are equivalent:
    \begin{enumerate}[(a)]
        \item $Q_\alpha^t(\phi_1)\mathbf{v}_1=Q_\alpha^t(\phi_2)\mathbf{v}_2$ for all $\alpha\in[0,1)$ and for all $t\in\N;$
        \item $Q_\alpha^t(\phi_1)\mathbf{v}_1=Q_\alpha^t(\phi_2)\mathbf{v}_2$ for all $t=0,1,\cdots,\min\left\{G_0(\phi_1,\mathbf{v}_1),G_0(\phi_2,\mathbf{v}_2)\right\}-1.$
    \end{enumerate}
\end{lem}

\begin{proof}It is obvious that
    (a) implies (b). Let us assume (b) and prove (a). According to Lemma~\ref{G}, $G_0(\phi_i,\mathbf{v}_i)=G_\alpha(\phi_i,\mathbf{v}_i)$ for $i=1,2.$ The proof, which we provide for completeness, is essentially the same as the proof of \cite[Proposition 2.2]{TPa}. Note that (b) and \eqref{defG} imply $G_0(\phi_1,\mathbf{v}_1)=G_0(\phi_2,\mathbf{v}_2).$ Let us define $g:=G_0(\phi_1,\mathbf{v}_1)=G_0(\phi_2,\mathbf{v}_2).$ 
      If $g=1$, then $\mathbf{v}_1=c\pmb{1}$ for some $c\in\R,$ which together with $\mathbf{v}_1=\mathbf{v}_2$ implies $Q_\alpha^t(\phi_1)\mathbf{v}_1=Q_\alpha^t(\phi_2)\mathbf{v}_2=\frac{c}{(1-\alpha)^t}\pmb{1}$ for all $t\in\N$. Thus (b) holds if $g=1.$
    
    Let $g\ge2.$ Similarly to the proof of Lemma \ref{G}, we define 
		\[B_\alpha:=\begin{bmatrix} \pmb{1} & \mathbf{v}_1 & Q_\alpha(\phi_1)\mathbf{v}_1 & \cdots & Q_\alpha^{g-2}(\phi_1)\mathbf{v}_1 \end{bmatrix}=\begin{bmatrix} \pmb{1} & \mathbf{v}_2 & Q_\alpha(\phi_2)\mathbf{v}_2 & \cdots & Q_\alpha^{g-2}(\phi_2)\mathbf{v}_2 \end{bmatrix}\in \R^{m\times g} ,\]
		 where $B_\alpha=\begin{bmatrix} \pmb{1} & \mathbf{v}_1 \end{bmatrix}=\begin{bmatrix} \pmb{1} & \mathbf{v}_2  \end{bmatrix}$ if $g=2.$ In view of (b), $\mathbf{v}_1=\mathbf{v}_2.$ Let us define $\mathbf{v}:=\mathbf{v}_1=\mathbf{v}_2.$
         
          Let us prove the equalities
		\begin{equation}\label{eqPTphipPTpsi} Q_\alpha^t(\phi_1)\mathbf{v}=Q_\alpha^t(\phi_2)\mathbf{v}=B_\alpha\mathbf{c}_t,\qquad \forall t\in\N, 
		\end{equation}
		for $\mathbf{c}_t\in\R^g,$ where the $\alpha$-dependent $c_{g-1}\in\R^g$ exists and unique in view of (b) and \eqref{defG},
		and
		\begin{equation}\label{eqctGHEF}\mathbf{c}_t=\begin{cases}
				\mathbf{e}_{t+2}&\ \text{if $t\le g-2,$}\\
				\begin{bmatrix} \frac{1}{1-\alpha}\mathbf{e}_1 & \mathbf{e}_3 & \mathbf{e}_4 & \cdots & \mathbf{e}_g & \mathbf{c}_{g-1}\\ \end{bmatrix}^{t-g+1}\mathbf{c}_{g-1}&\ \text{if $t\ge g-1.$}
		\end{cases}\end{equation}
		Here each $\mathbf{e}_i\in\R^g, i=1,\cdots,g,$ is the $i$-th column unit vector with the $i$-th entry being $1$ and all other entries being $0$. Equalities \eqref{eqPTphipPTpsi} imply (a).

        For $t=0,1,\ldots,g-2,$ equalities \eqref{eqPTphipPTpsi} follow from (b). For $t=g-1,$ \eqref{eqPTphipPTpsi} is already proven above. Suppose that equalities  \eqref{eqPTphipPTpsi}  hold for some $t\ge g-1$. Then, for $i=1,2,$
		\begin{align*}
			Q_\alpha^{t+1}(\phi_i)\mathbf{v}=Q_\alpha(\phi_i)Q_\alpha^t(\phi_i)\mathbf{v}&=Q_\alpha(\phi_i)B_\alpha\mathbf{c}_t\\
           &=\begin{bmatrix} \frac{1}{1-\alpha}\pmb{1} & Q_\alpha(\phi_i)\mathbf{v} & Q_\alpha^2(\phi_i)\mathbf{v} & \cdots & Q_\alpha^{g-1}(\phi_i)\mathbf{v} \end{bmatrix}\mathbf{c}_t\\
			&=B_\alpha\begin{bmatrix} \frac{1}{1-\alpha}\mathbf{e} & \mathbf{e}_3 & \mathbf{e}_4 & \cdots & \mathbf{e}_g & \mathbf{c}_{g-1}\\ \end{bmatrix}\mathbf{c}_t=B_\alpha\mathbf{c}_{t+1},
		\end{align*}
		where $\mathbf{c}_{t+1}=\begin{bmatrix} \frac{1}{1-\alpha}\mathbf{e}_1 & \mathbf{e}_3 & \mathbf{e}_4 & \cdots & \mathbf{e}_g & \mathbf{c}_{g-1}\\ \end{bmatrix}\mathbf{c}_t$ by definition.  Thus, \eqref{eqPTphipPTpsi} is proved by induction.
\end{proof}



Let $\norm{\cdot}$ be the norms in $\R^m$ or $\R^{m\times m}$ defined as
\[\norm{\mathbf{u}}:=\max_{x\in\X}\abs*{\mathbf{u}(x)},\ \mathbf{u}\in\R^m;\qquad\norm{P}:=\max_{1\le i\le m}\sum_{j=1}^m\abs*{P_{i,j}},\ P\in\R^{m\times m}.\]
Note that $\norm{P}=1$ if $P$ is a stochastic matrix, and $\norm{Q_\alpha^n(\phi)}=\frac{1}{(1-\alpha)^n}$ for all $\phi\in\F$ and $n\in\N.$ Also  $\norm{P\mathbf{u}}\le\norm{P}\norm{\mathbf{u}}$ for $P\in\R^{m\times m},\mathbf{u}\in\R^m$. Let $\normo{\cdot}$ be the seminorm in $\mathbb{R}^m$ defined as $\normo*{\mathbf{u}}:=\max_{x\in\X}\mathbf{u}(x)-\min_{x\in\X}\mathbf{u}(x).$ Note that $\normo{\mathbf{u}}=0$ iff $\mathbf{u}=c\mathbf{1}$ for some $c\in\R,$ and $\normo{\mathbf{u}}\le2\norm{\mathbf{u}}.$ For $\phi\in\F,$ let us define the \textbf{Dobrushin coefficient} of the policy $\phi$ as
\begin{equation}\label{gamma}
\gamma_\phi=1-\min_{x,y\in\X}\sum_{z\in\X}\min\left\{p(z|x,\phi(x)),p(z|y,\phi(y))\right\}.
\end{equation}
Clearly, $0\le\gamma_\phi\le1.$ It is known that $\gamma_\phi<1$ iff the Markov chain governed by the policy $\phi$ is uniformly ergodic \cite[Theorem 16.0.2]{Do}. Additional properties of $\gamma_\phi$ can be found on \cite[p.~198]{Put}. The following lemma is a generalization of \cite[Proposition 6.6.1]{Put} from $\alpha=0$ to an arbitrary $\alpha\in[0,1).$

\begin{lem}\label{lemQ} 
Let $\alpha\in[0,1),$ $\phi\in\F,$ $n\in\N,$ and $\mathbf{u}\in\R^m.$ Then $\normo*{Q_\alpha^n(\phi)\mathbf{u}}\le\left(\frac{\gamma_\phi}{1-\alpha\gamma_\phi}\right)^n\normo*{\mathbf{u}}.$ 
\end{lem}
\begin{proof} 
Let us prove the case for $n=1.$ If $\alpha=0$ then $\normo*{P(\phi)\mathbf{u}}\le\gamma_\phi\normo*{\mathbf{u}}$ in view of \cite[Proposition 6.6.1]{Put}. For $\alpha\in(0,1),$ note that
\begin{align*}
    \normo{Q_\alpha(\phi)\mathbf{u}}&=\normo{P(\phi)[I-\alpha P(\phi)]^{-1}\mathbf{u}}\le\gamma_\phi\normo{[I-\alpha P(\phi)]^{-1}\mathbf{u}}\\
    &=\gamma_\phi\normo*{\sum_{t=0}^\infty\alpha^tP^t(\phi)\mathbf{u}}\le\gamma_\phi\sum_{t=0}^
\infty\alpha^t\normo*{P^t(\phi)\mathbf{u}}\le\gamma_\phi\sum_{t=0}^
\infty(\alpha\gamma_\phi)^t\normo*{\mathbf{u}}=\frac{\gamma_\phi}{1-\alpha\gamma_\phi}\normo{\mathbf{u}},
\end{align*}
where the first and third inequalities are by the statement for $\alpha=0,$ and the second inequality follows from the triangle inequality for seminorms. Hence, the statement is true for $n=1.$ Suppose the statement is true for $n=k$ for some $k\ge1.$ Then for $n=k+1,$
\[\normo*{Q_\alpha^{k+1}(\phi)\mathbf{u}}=\normo*{Q_\alpha(\phi)Q_\alpha^k(\phi)\mathbf{u}}\le\gamma_\phi\normo*{Q_\alpha^k(\phi)\mathbf{u}}\le\gamma\left(\frac{\gamma_\phi}{1-\alpha\gamma_\phi}\right)^k\normo*{\mathbf{u}},\]
where the first inequality follows from the statement of $n=1,$ and the second inequality follows from the induction assumption. Therefore, the lemma is proved by induction.\qedhere

\end{proof}

The following lemma is an elementary result from analysis.

\begin{lem}\label{lemdiff}
    Let $\mathbf{u}_\theta\in\mathbb{R}^m$ be a function of $\theta$ mapping an open set $E\subset\R$ to $\R^m.$ Suppose $\frac{d}{d\theta}\mathbf{u}_\theta$ exists for some $\theta\in E.$ Let $\frac{d}{d\theta^-}(\cdot)$ and $\frac{d}{d\theta^+}(\cdot)$ denote the left derivative and right derivative, respectively. Then 
    \begin{enumerate}[(a)]
        \item both $\frac{d}{d\theta^-}\norm{\mathbf{u}_\theta}$ and $\frac{d}{d\theta^+}\norm{\mathbf{u}_\theta}$ exist, and $\max\left\{\abs*{\frac{d}{d\theta^-}\norm{\mathbf{u}_\theta}},\abs*{\frac{d}{d\theta^+}\norm{\mathbf{u}_\theta}}\right\}\le\norm*{\frac{d}{d\theta}\mathbf{u}_\theta}.$ In particular, if $\frac{d}{d\theta^-}\norm{\mathbf{u}_\theta}=\frac{d}{d\theta^+}\norm{\mathbf{u}_\theta},$ that is, if $\frac{d}{d\theta}\norm{\mathbf{u}_\theta}$ exists, then $\abs*{\frac{d}{d\theta}\norm{\mathbf{u}_\theta}}\le\norm*{\frac{d}{d\theta}\mathbf{u}_\theta};$
        \item both $\frac{d}{d\theta^-}\normo{\mathbf{u}_\theta}$ and $\frac{d}{d\theta^+}\normo{\mathbf{u}_\theta}$ exist, and $\max\left\{\abs*{\frac{d}{d\theta^-}\normo{\mathbf{u}_\theta}},\abs*{\frac{d}{d\theta^+}\normo{\mathbf{u}_\theta}}\right\}\le\normo*{\frac{d}{d\theta}\mathbf{u}_\theta}.$ In particular, if $\frac{d}{d\theta^-}\normo{\mathbf{u}_\theta}=\frac{d}{d\theta^+}\normo{\mathbf{u}_\theta},$ that is, if $\frac{d}{d\theta}\normo{\mathbf{u}_\theta}$ exists, then $\abs*{\frac{d}{d\theta}\normo{\mathbf{u}_\theta}}\le\normo*{\frac{d}{d\theta}\mathbf{u}_\theta}.$
    \end{enumerate}

\begin{proof}
    See Appendix~\ref{ap2}.
\end{proof}
\end{lem}


\section{Estimation of Objective Functions, their Derivatives, and Value Functions in Discount Factors}\label{seces}
For $n\in\N,$ $\alpha\in[0,1),$ and $\phi,\psi\in\F,$ let $u_{n,\alpha}^\phi:=\frac{1}{n!}\frac{d^n}{d\alpha^n}v_\alpha^\phi.$ Observe that $u_{0,\alpha}^\phi=v_\alpha^\phi.$
The following lemma provides formulas for derivatives of objective functions for policies. It also implies that any fixed order of derivatives of objective functions can be computed in strongly polynomial time.

\begin{lem}\label{diff} Let $\alpha\in[0,1),$ $n\in\N,$ and $\phi\in\F.$ Then $u_{n,\alpha}^\phi=Q_\alpha^n(\phi)v_\alpha^\phi.$
\end{lem}
\begin{proof}
    Since the matrix $\left[I-\alpha P(\phi)\right]$ is invertable and $v_\alpha^\phi=\left[I-\alpha P(\phi)\right]^{-1}r(\phi),$ we have
    \[u_{n,\alpha}^\phi=\frac{1}{n!}\frac{d^n}{d\alpha^n}v_\alpha^\phi=\frac{1}{n!}\frac{d^n}{d\alpha^n}\left[I-\alpha P(\phi)\right]^{-1}r(\phi)=P^n(\phi)\left[I-\alpha P(\phi)\right]^{-{(n+1)}}r(\phi)=Q_\alpha^n(\phi)v_\alpha^\phi. \qedhere\]
\end{proof}

The following corollary provides necessary and sufficient conditions for two policies to have equal objective functions for all discount factors. It is also a generalization of \cite[Theorem 2.1]{TPa} from $\alpha=0$ to $\alpha\in[0,1).$

\begin{cor}\label{lemeq1} Let $\left\{\phi,\psi\right\}\subset\F$. Then  $v_\alpha^\phi=v_\alpha^{\psi}$ for all $\alpha\in[0,1)$ if and only if
 there exists $\alpha\in[0,1)$ such that $Q_\alpha^k(\phi)v_\alpha^\phi=Q_\alpha^k(\psi)v_\alpha^\psi$ for all $k=0,1,\cdots,\min\{G_0(\phi,v_\alpha^\phi),G_0(\psi,v_\alpha^\psi)\}-1.$
\end{cor}
\begin{proof} By Lemma~\ref{lempartition}(a) the vector functions $v_\alpha^\phi$ and $v_\alpha^{\psi}$ are analytic on $\alpha\in[0,1).$ By the Taylor's expansion and by Lemma~\ref{diff}(a),
\begin{equation}\label{QT}
v_\beta^\phi=\sum_{k=0}^\infty\frac{(\beta-\alpha)^k}{k!}\frac{d^k}{d\alpha^k}v_\alpha^\phi=\sum_{k=0}^\infty(\beta-\alpha)^kQ_{\alpha}^k(\phi)v_\alpha^\phi\qquad\forall\alpha,\beta\in[0,1),
\end{equation}
and the proof follows from Theorem~\ref{la}.
\end{proof}

\begin{cor}\label{Taylor} Let $\alpha,\beta\in[0,1)$ and $\phi\in\F.$ Then $v_\beta^\phi=\sum_{k=0}^\infty(\beta-\alpha)^ku_{k,\alpha}^\phi.$
\end{cor}
\begin{proof}
    This is by equation~\eqref{QT} and by Lemma~\ref{diff}(a).
\end{proof}

For objective functions with a fixed discount factor, the following corollary provides upper bounds for norms and seminorms of higher-order derivatives in terms of their lower-order derivatives.

\begin{thm}\label{diffb} Let $\phi\in\F$ and $n\in\N.$ Then for all $k=0,1,2,\cdots,n$ and for all $\alpha\in[0,1),$ (a) $\norm*{u_{n,\alpha}^\phi}\le\left(\frac{1}{1-\alpha}\right)^{n-k}\norm*{u_{k,\alpha}^\phi};$ (b) $\normo*{u_{n,\alpha}^\phi}\le\left(\frac{\gamma_\phi}{1-\alpha\gamma_\phi}\right)^{n-k}\normo*{u_{k,\alpha}^\phi}.$
\end{thm}
\begin{proof}
    By Lemma~\ref{diff},
\[\norm*{u_{n,\alpha}^\phi}=\norm*{Q_\alpha^{n-k}(\phi)Q_\alpha^k(\phi)v_\alpha^\phi}\le\norm*{Q_\alpha^{n-k}(\phi)}\norm*{Q_\alpha^k(\phi)v_\alpha^\phi}=\left(\frac{1}{1-\alpha}\right)^{n-k}\norm*{u_{k,\alpha}^\phi}.\]
Similarly, by Lemma~\ref{lemQ}(b),
\[\normo*{u_{n,\alpha}^\phi}=\normo*{Q_\alpha^{n-k}(\phi)Q_\alpha^k(\phi)v_\alpha^\phi}\le\left(\frac{\gamma_\phi}{1-\alpha\gamma_\phi}\right)^{n-k}\normo*{Q_\alpha^k(\phi)v_\alpha^\phi}=\left(\frac{\gamma_\phi}{1-\alpha\gamma_\phi}\right)^{n-k}\normo*{u_{k,\alpha}^\phi}.\qedhere\]
\end{proof}

The following theorem reveals the monotonicity in discount factors of the norm and the seminorm of derivatives of objective functions.
\begin{thm}\label{mon} Let $n\in\N$ and $\phi\in\F.$ Then
    \begin{enumerate}[(a)]
        \item both $(1-\theta)^{n+1}\norm*{u_{n,\theta}^\phi}$ and $(1-\theta\gamma_\phi)^{n+1}\normo*{u_{n,\theta}^\phi}$ are non-increasing in $\theta\in[0,1);$
        \item both $(1-\theta)^{-(n+1)}\norm*{u_{n,\theta}^\phi}$ and $(1-\theta\gamma_\phi)^{-(n+1)}\normo*{u_{n,\theta}^\phi}$ are non-decreasing in $\theta\in[0,1).$
    \end{enumerate}
\end{thm}
\begin{proof} By Lemma~\ref{lempartition}(a) we know that $\norm*{u_{n,\theta}^\phi}$ is a finitely piecewise  rational function of $\theta,$ as the maximum norm of a rational vector function. Thus, there exist $j\in\N^+$ and points  $0=c_0<c_1<\ldots<c_j<c_{j+1}=1$ such that $\norm*{u_{n,\theta}^\phi}$ is analytic on $\theta\in(c_i,c_{i+1})$ for each $0\le i\le j.$ Let $f(\theta)=(1-\theta)^{n+1}\norm*{u_{n,\theta}^\phi}.$ Then Lemma~\ref{lemdiff}(a) and Theorem~\ref{diffb}(a) imply, for $i=0,1,\ldots,j$ and for $\theta\in(c_i,c_{i+1}),$
    \begin{align*}
        f'(\theta)&=(1-\theta)^{n+1}\left(-\frac{n+1}{1-\theta}\norm*{u_{n,\theta}^\phi}+\frac{d}{d\theta}\norm*{u_{n,\theta}^\phi}\right)\le(1-\theta)^{n+1}\left(-\frac{n+1}{1-\theta}\norm*{u_{n,\theta}^\phi}+\norm*{\frac{d}{d\theta}u_{n,\theta}^\phi}\right)\\
        &=(n+1)(1-\theta)^{n+1}\left(-\frac{1}{1-\theta}\norm*{u_{n,\theta}^\phi}+\norm*{u_{n+1,\theta}^\phi}\right)\le0. 
    \end{align*}
Since $f(\theta)$ is continuous,  $f(\theta)$ is non-increasing on $[0,1).$ On the other hand, let $g(\theta)=(1-\theta)^{-(n+1)}\norm*{u_{n,\theta}^\phi}.$ Again by Lemma~\ref{lemdiff}(a) and by Theorem~\ref{diffb}(a), for $i=0,1,\ldots,j$ and for $\theta\in(c_i,c_{i+1}),$
\begin{align*}
    g'(\theta)&=\frac{1}{(1-\theta)^{n+1}}\left(\frac{n+1}{1-\theta}\norm*{u_{n,\theta}^\phi}+\frac{d}{d\theta}\norm*{u_{n,\theta}^\phi}\right)\ge\frac{1}{(1-\theta)^{n+1}}\left(\frac{n+1}{1-\theta}\norm*{u_{n,\theta}^\phi}-\norm*{\frac{d}{d\theta}u_{n,\theta}^\phi}\right)\\
    &=\frac{n+1}{(1-\theta)^{n+1}}\left(\frac{1}{1-\theta}\norm*{u_{n,\theta}^\phi}-\norm*{u_{n+1,\theta}^\phi}\right)\ge0. 
\end{align*}
Since $g(\theta)$ is continuous,  $g(\theta)$ is non-decreasing on $[0,1).$ By Lemma~\ref{lempartition}(a),  $\normo*{u_{n,\theta}^\phi}$ is  a finitely piecewise rational function of $\theta.$ Thus, there exist $k\in\N^+$ and points $0=d_0<d_1<\ldots<d_k<d_{k+1}=1$ such that $\normo*{u_{n,\theta}^\phi}$ is analytic on $\theta\in(d_i,d_{i+1})$ for each $0\le i\le k.$ Let $f_*(\theta)=(1-\theta\gamma_\phi)^{n+1}\normo*{u_{n,\theta}^\phi}.$ Then Lemma~\ref{lemdiff}(b) and Theorem~\ref{diffb}(b) imply, for $i=0,1,\ldots,k$ and for $\theta\in(d_i,d_{i+1}),$
    \begin{align*}
        f_*'(\theta)&=(1-\theta\gamma_\phi)^{n+1}\left[-\frac{(n+1)\gamma_\phi}{1-\theta\gamma_\phi}\normo*{u_{n,\theta}^\phi}+\frac{d}{d\theta}\normo*{u_{n,\theta}^\phi}\right]\\
        &\le(1-\theta\gamma_\phi)^{n+1}\left[-\frac{(n+1)\gamma_\phi}{1-\theta\gamma_\phi}\normo*{u_{n,\theta}^\phi}+\normo*{\frac{d}{d\theta}u_{n,\theta}^\phi}\right]\\
        &=(n+1)(1-\theta\gamma_\phi)^{n+1}\left(-\frac{\gamma_\phi}{1-\theta\gamma_\phi}\normo*{u_{n,\theta}^\phi}+\normo*{u_{n+1,\theta}^\phi}\right)\le0.
    \end{align*}
Since $f_*(\theta)$ is continuous,  $f_*(\theta)$ is non-increasing on $[0,1).$ On the other hand, let $g_*(\theta)=(1-\theta\gamma_\phi)^{-(n+1)}\normo*{u_{n,\theta}^\phi}.$ Again by Lemma~\ref{lemdiff}(b) and by Theorem~\ref{diffb}(b), for $i=0,1,\ldots,k$ and for $\theta\in(d_i,d_{i+1}),$
\begin{align*}
    g_*'(\theta)&=\frac{1}{(1-\theta\gamma_\phi)^{n+1}}\left[\frac{(n+1)\gamma_\phi}{1-\theta\gamma_\phi}\normo*{u_{n,\theta}^\phi}+\frac{d}{d\theta}\normo*{u_{n,\theta}^\phi}\right]\\
    &\ge\frac{1}{(1-\theta\gamma_\phi)^{n+1}}\left[\frac{(n+1)\gamma_\phi}{1-\theta\gamma_\phi}\normo*{u_{n,\theta}^\phi}-\normo*{\frac{d}{d\theta}u_{n,\theta}^\phi}\right]\\
    &=\frac{n+1}{(1-\theta\gamma_\phi)^{n+1}}\left(\frac{\gamma_\phi}{1-\theta\gamma_\phi}\normo*{u_{n,\theta}^\phi}-\normo*{u_{n+1,\theta}^\phi}\right)\ge0.
\end{align*}
Since $g_*(\theta)$ is continuous,  $g_*(\theta)$ is non-decreasing on $[0,1).$
\end{proof}

The following corollary, which is equivalent to Theorem~\ref{mon}, provides upper bounds for the norm and the seminorm of derivatives of objective functions at a discount factor in terms of their same-order derivatives at nearby discount factors.
\begin{cor}\label{vb}
    Let $n\in\N,$ $0\le\alpha\le\theta\le\beta<1,$ and $\phi\in\F.$ Then
    \begin{enumerate}[(a)]
        \item $\norm*{u_{n,\theta}^\phi}\le\min\left\{\left(\frac{1-\alpha}{1-\theta}\right)^{n+1}\norm*{u_{n,\alpha}^\phi}, \left(\frac{1-\theta}{1-\beta}\right)^{n+1}\norm*{u_{n,\beta}^\phi}\right\};$
        \item $\normo*{u_{n,\theta}^\phi}\le\min\left\{\left(\frac{1-\alpha\gamma_\phi}{1-\theta\gamma_\phi}\right)^{n+1}\normo*{u_{n,\alpha}^\phi}, \left(\frac{1-\theta\gamma_\phi}{1-\beta\gamma_\phi}\right)^{n+1}\normo*{u_{n,\beta}^\phi}\right\}.$
    \end{enumerate}
\end{cor}

\begin{proof} 
   These inequalities follow directly from the monotonicity indicated in Theorem~\ref{mon}.
\end{proof}

The following corollary, which is also equivalent to Theorem~\ref{mon}, provides bounds for the ratios of the norms and seminorms of derivatives of objective functions for two discount factors.
\begin{cor}\label{ratiob} Let $\phi\in\F,$ $n\in\N,$ and $\alpha,\beta\in[0,1).$
    \begin{enumerate}[(a)]
        \item If $\norm*{u_{n,\alpha}^\phi}\ne0,$ then $\norm*{u_{n,\beta}^\phi}\ne0,$ and $\left(\frac{1-\max\{\alpha,\beta\}}{1-\min\{\alpha,\beta\}}\right)^{n+1}\le\frac{\norm{u_{n,\beta}^\phi}}{\norm{u_{n,\alpha}^\phi}}\le\left(\frac{1-\min\{\alpha,\beta\}}{1-\max\{\alpha,\beta\}}\right)^{n+1};$ 
        \item If $\normo*{u_{n,\alpha}^\phi}\ne0,$ then $\normo*{u_{n,\beta}^\phi}\ne0,$ and $\left(\frac{1-\gamma_\phi\max\{\alpha,\beta\}}{1-\gamma_\phi\min\{\alpha,\beta\}}\right)^{n+1}\le\frac{\normo{u_{n,\beta}^\phi}}{\normo{u_{n,\alpha}^\phi}}\le\left(\frac{1-\gamma_\phi\min\{\alpha,\beta\}}{1-\gamma_\phi\max\{\alpha,\beta\}}\right)^{n+1}.$
    \end{enumerate}
\end{cor}
\begin{proof}
    This corollary follows directly from Theorem~\ref{mon}.
\end{proof}

The following corollary shows that, if at a discount factor some derivative of an objective function has zero norm or zero seminorm, then all its higher-order derivatives also have zero norm or zero seminorm for all discount factors.
\begin{cor}\label{diff0}
    Let $\phi\in\F.$ 
    \begin{enumerate}[(a)]
        \item If $\norm*{u_{k,\alpha}^\phi}=0$ for some $\alpha\in[0,1)$ and for some $k\in\N,$ then $\norm*{u_{n,\alpha}^\phi}=0$ for all $n\ge k$ and for all $\alpha\in[0,1);$
        \item if $\normo*{u_{k,\alpha}^\phi}=0$ for some $\alpha\in[0,1)$ and for some $k\in\N,$ then $\normo*{u_{n,\alpha}^\phi}=0$ for all $n\ge k$ and for all $\alpha\in[0,1).$
    \end{enumerate}
\end{cor}
\begin{proof} This corollary follows from Theorem~\ref{diffb} and Theorem~\ref{mon}.
\end{proof}

\begin{cor}\label{zeronorm}
    Let $\phi\in\F$ and $k\in\N.$ Then 
    \begin{enumerate}[(a)]
        \item $\norm*{u_{k,\alpha}^\phi}=0$ for some $\alpha\in[0,1)$ iff $P^t(\phi)r(\phi)=\mathbf{0}$ for all $t\ge\min\{k,m-1\};$
        \item $\normo*{u_{k,\alpha}^\phi}=0$ for some $\alpha\in[0,1)$ iff there exists $c\in\R$ such that $P^t(\phi)r(\phi)=c\mathbf{1}$ for all $t\ge\min\{k,m-1\}.$ 
    \end{enumerate}
\end{cor}
\begin{proof}
    Let us prove (a). By Lemma~\ref{lem1}(a) we only need to prove that $\norm*{u_{k,\alpha}^\phi}=0$ for some $\alpha\in[0,1)$ iff $P^k(\phi)r(\phi)=\mathbf{0}.$ The sufficiency is true by the definition of $u_{k,\alpha}^\phi$ and by $v_\alpha^\phi=\sum_{t=0}^\infty\alpha^tP^t(\phi)r(\phi)=\sum_{t=0}^{\min\left\{k,m-1\right\}-1}\alpha^tP^t(\phi)r(\phi).$ For the necessity, note that $\norm*{u_{k,\alpha}^\phi}=0$ implies $\norm*{u_{k,0}^\phi}=0$ by Corollary~\ref{diff0}(a), and $P^k(\phi)r(\phi)=u_{k,0}^\phi=\mathbf{0}.$
    

    Let us prove (b). By Lemma~\ref{lem1}(b) we only need to prove that $\normo*{u_{k,\alpha}^\phi}=0$ for some $\alpha\in[0,1)$ iff $P^k(\phi)r(\phi)=c\mathbf{1}$ for some $c\in\R.$ Let $j=\min\left\{k,m-1\right\}.$ The sufficiency is true by the definition of $u_{k,\alpha}^\phi$ and by \[v_\alpha^\phi=\sum_{t=0}^\infty\alpha^tP^t(\phi)r(\phi)=\sum_{t=0}^{j-1}\alpha^tP^t(\phi)r(\phi)+\sum_{t=j}^{\infty}\alpha^tP^t(\phi)r(\phi)=\sum_{t=0}^{j-1}\alpha^tP^t(\phi)r(\phi)+\frac{c\alpha^j}{1-\alpha}\mathbf{1}.\]
    For the necessity, note that $\normo*{u_{k,\alpha}^\phi}=0$ implies $\normo*{u_{k,0}^\phi}=0$ by Corollary~\ref{diff0}(b), which means $u_{k,0}^\phi=c\mathbf{1}$  for some $c\in\R.$ Thus, $P^k(\phi)r(\phi)=u_{k,0}^\phi=c\mathbf{1}.$
\end{proof}

For $\phi\in\F,$ recall that $\gamma_\phi$ is defined in \eqref{gamma}. Let $\gamma:=\max_{\phi\in\F}\gamma_\phi.$ Then
    \[\gamma=\max_{x,y\in\X}\ \  \max_{ a\in A_x,\ b\in  A_y}\left[1-\sum_{z\in\X}\min\left\{p(z|x,a),p(z|y,b)\right\}\right]\]
    because the internal maximum won't change if we add the additional restriction $a=b$ if $x=y.$
    
The following theorem states monotonicity of the value function in discount factors similar to that in Theorem~\ref{mon}.
\begin{thm}\label{Vine}
\begin{enumerate}[(a)]
        \item both $(1-\theta)\norm*{V_\theta}$ and $(1-\theta\gamma)\normo*{V_\theta}$ are non-increasing in $\theta\in[0,1);$
        \item both $(1-\theta)^{-1}\norm*{V_\theta}$ and $(1-\theta\gamma)^{-1}\normo*{V_\theta}$ are non-decreasing in $\theta\in[0,1).$
    \end{enumerate}
\end{thm}
\begin{proof}
    This theorem follows by letting $n=0$ in Theorem~\ref{mon} and by Lemma~\ref{lempartition}(a).
\end{proof}

The following corollary is similar to Corollary~\ref{vb}, but it deals with value functions.
\begin{cor}\label{valuein}
    Let $n\in\N,$ $0\le\alpha\le\theta\le\beta<1,$ and $\phi\in\F.$ Then
    \begin{enumerate}[(a)]
        \item $\norm*{V_\theta}\le\min\left\{\frac{1-\alpha}{1-\theta}\norm*{V_\alpha}, \frac{1-\theta}{1-\beta}\norm*{V_\beta}\right\};$
        \item $\normo*{V_\theta}\le\min\left\{\frac{1-\alpha\gamma}{1-\theta\gamma}\normo*{V_\alpha}, \frac{1-\theta\gamma}{1-\beta\gamma}\normo*{V_\beta}\right\}.$
    \end{enumerate}
\end{cor}
\begin{proof}
        This corollary follows from Theorem~\ref{Vine}.
\end{proof}

The following corollary is similar to Corollary~\ref{ratiob}, but it deals with value functions.
\begin{cor}\label{V0} Let $\alpha,\beta\in[0,1).$
    \begin{enumerate}[(a)]
        \item If $\norm*{V_\alpha}\ne0,$ then $\norm*{V_\beta}\ne0,$ and $\frac{1-\max\{\alpha,\beta\}}{1-\min\{\alpha,\beta\}}\le\frac{\norm{V_\beta}}{\norm{V_\alpha}}\le\frac{1-\min\{\alpha,\beta\}}{1-\max\{\alpha,\beta\}};$ 
        \item If $\normo*{V_\alpha}\ne0,$ then $\normo*{V_\beta}\ne0,$ and  $\frac{1-\gamma\max\{\alpha,\beta\}}{1-\gamma\min\{\alpha,\beta\}}\le\frac{\normo{V_\beta}}{\normo{V_\alpha}}\le\frac{1-\gamma\min\{\alpha,\beta\}}{1-\gamma\max\{\alpha,\beta\}};$ 
    \end{enumerate}
\end{cor}
\begin{proof}
    This corollary follows from Theorem~\ref{Vine}.
\end{proof}

The following corollary is similar to Corollary~\ref{diff0}, but it deals with value functions. It states that, if the norm or the seminorm of the value function is $0$ at some discount factor, then its norm or its seminorm is $0$ for all discount factors.
\begin{cor}\label{v0}
    \begin{enumerate}[(a)]
        \item $\norm{V_\alpha}=0$ for some $\alpha\in[0,1)$ iff $\norm{V_\beta}=0$ for all $\beta\in[0,1);$
        \item $\normo{V_\alpha}=0$ for some $\alpha\in[0,1)$ iff $\normo{V_\beta}=0$ for all $\beta\in[0,1).$
    \end{enumerate}
\end{cor}
\begin{proof} This  corollary follows from Theorem~\ref{Vine} or Corollary~\ref{V0}.
\end{proof}
Let $R_x=\max_{a\in A_x}r(x,a)$ and $R=\max_{x\in\X}R_x.$ We recall that $D(0)=\left\{\phi\in\F:r(x,\phi(x))\right\}.$
    \begin{cor}\label{Value0} 
\begin{enumerate}[(a)] 
    \item $\norm{V_\alpha}=0$ for some $\alpha\in[0,1)$ iff $R_x=0$ for all $x\in\X;$
    \item $\normo{V_\alpha}=0$ for some $\alpha\in[0,1)$ iff there is no irregular point, $D(\beta)=D(0)$ and $V_\beta=\frac{R}{1-\beta}\mathbf{1}$ for all $\beta\in[0,1).$ In particular, if $\alpha$ is an irregular point, then $\normo{V_\alpha}>0,$ i.e., all entries of $V_\alpha$ cannot be identical.
\end{enumerate}
\end{cor}
\begin{proof}
    By Corollary~\ref{v0}(a) $\norm{V_\alpha}=0$ iff $\norm{V_0}=0,$ and $V_0(x)=R_x$ for all $x\in\X.$ This proves (a). Sufficiency of (b) is obvious. Let us prove the necessity of (b). If $\normo{V_\alpha}=0,$ then $\normo{V_0}=0$ by Corollary~\ref{v0}(b), which implies $V_0=r(\phi)=R\mathbf{1}$ for all $\phi\in D(0).$ Then $v_\beta^\phi=\frac{R}{1-\beta}\mathbf{1}$ for all $\phi\in D(0)$ and for all $\beta\in[0,1).$ However, it is well-known that $\norm{V_\beta}\le\frac{R}{1-\beta}$ for all $\beta\in[0,1)$ \cite[p.~163]{Put}. Thus $V_\beta=\frac{R}{1-\beta}\mathbf{1}$ and $D(0)\subset D(\beta)$ for all $\beta\in[0,1),$ which also implies $0$ is a regular point. Moreover, suppose $\tilde{\alpha}\in(0,1)$ is an irregular point, i.e., there exists $\psi\notin D(0)$ such that $v_{\tilde{\alpha}}^\psi=V_{\tilde{\alpha}}=\frac{R}{1-\tilde{\alpha}}\mathbf{1}.$ Then $\normo*{v_{\tilde{\alpha}}^\psi}=\normo*{V_{\tilde{\alpha}}}=0.$ By Corollary~\ref{diff0}(b) $\normo*{v_0^\psi}=0,$ which together with $\psi\notin D(0)$ implies $r(\psi)=c\mathbf{1}$ for some $c<R.$ Then $v_{\tilde{\alpha}}^\psi=\frac{c}{1-\tilde{\alpha}}\mathbf{1},$ which contradicts to $v_{\tilde{\alpha}}^\psi=\frac{R}{1-\tilde{\alpha}}\mathbf{1}.$ Therefore, such an irregular point $\tilde{\alpha}$ doesn't exist, and $D(0)=D(\beta)$ for all $\beta\in[0,1).$
\end{proof}


\section{Bounds on Increments of Functions of Discount Factors}\label{seclip}
In this section we are going to use the following integral properties: if $\mathbf{u}_\theta\in\R^m$ is a vector function of $\theta$ such that $\mathbf{u}_\theta(x)$ is an integrable function over $\theta\in[a,b]\subset\R$ for each $x\in\X,$ then both $\norm{\mathbf{u}_\theta}$ and $\normo{\mathbf{u}_\theta}$ are integrable over $[a,b],$ and (a) $\norm*{\int_a^b\mathbf{u}_\theta d\theta}\le\int_a^b\norm{\mathbf{u}_\theta}d\theta;$ (b) $\normo*{\int_a^b\mathbf{u}_\theta d\theta}\le\int_a^b\normo{\mathbf{u}_\theta}d\theta.$ Integrability of $\norm{\mathbf{u}_\theta}$ and $\normo{\mathbf{u}_\theta}$ is by the definitions of $\norm{\cdot},$ $\normo{\cdot}$ and because maximums, minimums, and sums of two integrable functions are integrable functions; see, e.g., \cite{Real}. Statement (a) is well-known and follows from Minkowski's integral inequality. Statement (b) is true because we can let $z_1,z_2\in\X$ such that $\int_a^b\mathbf{u}_\theta(z_1)d\theta=\max_{x\in\X}\int_a^b\mathbf{u}_\theta(x)d\theta$ and $\int_a^b\mathbf{u}_\theta(z_2)d\theta=\min_{x\in\X}\int_a^b\mathbf{u}_\theta(x)d\theta,$ which implies
\[\normo*{\int_a^b\mathbf{u}_\theta d\theta}=
    \max_{x\in\X}\int_a^b\mathbf{u}_\theta(x)d\theta-\min_{x\in\X}\int_a^b\mathbf{u}_\theta(x)d\theta=\int_a^b\left[\mathbf{u}_\theta(z_1)-\mathbf{u}_\theta(z_2)\right]d\theta\le\int_a^b\normo{\mathbf{u}_\theta}d\theta.\]


\cite[Lemma 4.2]{TPa} shows that value functions are Lipschitz continuous. The following theorem extends the Lipschitz continuity to the derivatives of objective functions.

\begin{thm}\label{olip} Let $n\in\N,$ $\phi\in\F,$ and $\alpha,\beta\in[0,1).$ Then
\begin{enumerate}[(a)]
    \item $\norm*{u_{n,\beta}^\phi-u_{n,\alpha}^\phi}\le\left[\left(\frac{1-\min\{\alpha,\beta\}}{1-\max\{\alpha,\beta\}}\right)^{n+1}-1\right]\min\left\{\norm*{u_{n,\alpha}^\phi},\norm*{u_{n,\beta}^\phi}\right\};$
    \item $\normo*{u_{n,\beta}^\phi-u_{n,\alpha}^\phi}\le\left[\left(\frac{1-\gamma_\phi\min\{\alpha,\beta\}}{1-\gamma_\phi\max\{\alpha,\beta\}}\right)^{n+1}-1\right]\min\left\{\normo*{u_{n,\alpha}^\phi},\normo*{u_{n,\beta}^\phi}\right\}.$
\end{enumerate}
\end{thm}
\begin{proof}
    Without loss of generality, let us assume $\alpha\le\beta.$ 
    To prove (a), 
    let us observe that
        \begin{equation}\label{eqvint1}
\begin{split}
\norm*{u_{n,\beta}^\phi-u_{n,\alpha}^\phi} 
&= \norm*{\int_\alpha^\beta\frac{d}{d\theta}u_{n,\theta}^\phi d\theta} \le \int_\alpha^\beta\norm*{\frac{d}{d\theta}u_{n,\theta}^\phi}d\theta \\
&= \int_\alpha^\beta(n+1)\norm*{u_{n+1,\theta}^\phi}d\theta \le \int_\alpha^\beta\frac{n+1}{1-\theta}\norm*{u_{n,\theta}^\phi}d\theta,
\end{split}
\end{equation}
where the last inequality follows from $\norm*{u_{n+1,\theta}^\phi}\le (1-\theta)^{-1}\norm*{u_{n,\theta}^\phi},$ which holds by  Theorem~\ref{diffb}(a). Corollary~\ref{vb}(a) means that, for $\theta\in[\alpha,\beta],$
  \begin{align}
    \norm*{u_{n,\theta}^\phi}&\le\left(\frac{1-\alpha}{1-\theta}\right)^{n+1}\norm*{u_{n,\alpha}^\phi},\label{eqvintF1}\\
    \norm*{u_{n,\theta}^\phi}&\le\left(\frac{1-\theta}{1-\beta}\right)^{n+1}\norm*{u_{n,\beta}^\phi}.\label{eqvintF2}
  \end{align}
Formulae \eqref{eqvint1} and \eqref{eqvintF1} imply
   \begin{align}
        \norm*{u_{n,\beta}^\phi-u_{n,\alpha}^\phi}
        \le\int_\alpha^\beta\frac{(n+1)(1-\alpha)^{n+1}}{(1-\theta)^{n+2}}\norm*{u_{n,\alpha}^\phi}d\theta=\left[\left(\frac{1-\alpha}{1-\beta}\right)^{n+1}-1\right]\norm*{u_{n,\alpha}^\phi}.\label{eqvint2}
    \end{align}  
Formulae \eqref{eqvint1} and \eqref{eqvintF2} imply 
    \begin{equation}
        \norm*{u_{n,\beta}^\phi-u_{n,\alpha}^\phi}\le\int_\alpha^\beta\frac{(n+1)(1-\theta)^n}{(1-\beta)^{n+1}}\norm*{u_{n,\beta}^\phi}d\theta=\left[\left(\frac{1-\alpha}{1-\beta}\right)^{n+1}-1\right]\norm*{u_{n,\beta}^\phi}.\label{eqvint3}
    \end{equation}
    Formulae \eqref{eqvint2} and \eqref{eqvint3} together imply (a). 
    
    The proof of (b) is similar. Corollary~\ref{vb}(b) means that, for $\theta\in[\alpha,\beta],$ 
  \begin{align}
    \normo*{u_{n,\theta}^\phi}&\le\left(\frac{1-\alpha\gamma_\phi}{1-\theta\gamma_\phi}\right)^{n+1}\normo*{u_{n,\alpha}^\phi},\label{eqvintF3}\\
    \normo*{u_{n,\theta}^\phi}&\le\left(\frac{1-\theta\gamma_\phi}{1-\beta\gamma_\phi}\right)^{n+1}\normo*{u_{n,\beta}^\phi}. \label{eqvintF4}
  \end{align}
Next we observe that
     \begin{align}
        \normo*{u_{n,\beta}^\phi-u_{n,\alpha}^\phi}&=\normo*{\int_\alpha^\beta\frac{d}{d\theta}u_{n,\theta}^\phi d\theta}\nonumber\\
        &\le\int_\alpha^\beta\normo*{\frac{d}{d\theta}u_{n,\theta}^\phi}d\theta=\int_\alpha^\beta(n+1)\normo*{u_{n+1,\theta}^\phi}d\theta\le\int_\alpha^\beta\frac{(n+1)\gamma_\phi}{1-\theta\gamma_\phi}\normo*{u_{n,\theta}^\phi}d\theta,\label{eqvint4}
    \end{align}
where the last inequality follows from $\normo*{u_{n+1,\theta}^\phi}\le \left(\frac{\gamma_\phi}{1-\theta\gamma_\phi}\right)\normo*{u_{n,\theta}^\phi},$ which holds by Theorem~\ref{diffb}(b). Formulae \eqref{eqvint4} and \eqref{eqvintF3} imply
  
    \begin{align}
        \normo*{u_{n,\beta}^\phi-u_{n,\alpha}^\phi}
        \le\int_\alpha^\beta\frac{(n+1)(1-\alpha\gamma_\phi)^{n+1}\gamma_\phi}{(1-\theta\gamma_\phi)^{n+2}}\normo*{u_{n,\alpha}^\phi}d\theta=\left[\left(\frac{1-\alpha\gamma_\phi}{1-\beta\gamma_\phi}\right)^{n+1}-1\right]\normo*{u_{n,\alpha}^\phi}.\label{eqvint5}
    \end{align}
     Formulae \eqref{eqvint4} and \eqref{eqvintF4} imply
    \begin{equation}
        \normo*{u_{n,\beta}^\phi-u_{n,\alpha}^\phi}\le\int_\alpha^\beta\frac{(n+1)\gamma_\phi(1-\theta\gamma_\phi)^n}{(1-\beta\gamma_\phi)^{n+1}}\normo*{u_{n,\beta}^\phi}d\theta=\left[\left(\frac{1-\alpha\gamma_\phi}{1-\beta\gamma_\phi}\right)^{n+1}-1\right]\normo*{u_{n,\beta}^\phi}.\label{eqvint6}
    \end{equation}
    Formulae \eqref{eqvint5} and \eqref{eqvint6} together imply (b). 
\end{proof}

\begin{rem} In Theorem~\ref{olip}, the multiplier in inequality (a) does not exceed the multiplier in inequality (b), which is modified by $\gamma_\phi.$ We observe that inequality (a) does not hold if the multiplier in it is replaced with the multiplier from (b). Indeed, let us consider a policy $\phi$ whose one-step rewards are all identical, i.e., $r(x,\phi(x))=r(y,\phi(y))$ for all $x,y\in\X.$ Then the inequality in Theorem~\ref{olip}(a) becomes sharp for all $n\in\N$ meaning that in general the inequality in Theorem~\ref{olip}(b) does not hold for norms if $\gamma_\phi<1.$ 
\end{rem}
Statement (a) of the following theorem strengthens the Lipschitz continuity of value functions in \cite[Lemma 4.2]{TPa} and \cite[Theorem 2]{NIPS}.
\begin{thm}\label{Vlip} Let $\alpha,\beta\in[0,1).$ Then
\begin{enumerate}[(a)]
    \item $\norm{V_\beta-V_\alpha}\le\frac{\abs{\beta-\alpha}}{1-\max\{\alpha,\beta\}}\min\{\norm{V_\alpha},\norm{V_\beta}\};$
    \item $\normo{V_\beta-V_\alpha}\le\frac{\gamma\abs{\beta-\alpha}}{1-\gamma\max\{\alpha,\beta\}}\min\{\normo{V_\alpha},\normo{V_\beta}\}.$
\end{enumerate}
\end{thm}
\begin{proof}
    Without loss of generality, let us assume $\alpha\le\beta.$ By Lemma~\ref{lempartition}(a), there exist let $j\in\N^+$ points $\alpha=c_0<c_1<\ldots<c_j<c_{j+1}=\beta$ such that $c_1,c_2,\cdots,c_j$ are irregular points, and $V_\theta$ is analytic on $\theta\in(c_i,c_{i+1})$ for each $0\le i\le j.$ By Lemma~\ref{lempartition}(a), for each $i=0,1,\cdots,j,$ there is some $\phi_i\in\F$ such that $V_\theta=v_\theta^{\phi_i}$ for all $\theta\in[c_i,c_{i+1}].$

To prove (a), we observe that, in view of Theorem~\ref{olip}(a),
 \begin{equation}\label{VsubEF1}
        \norm*{V_\beta-V_\alpha}\le\sum_{i=0}^j\norm*{V_{c_{i+1}}-V_{c_i}}\le \sum_{i=0}^j\frac{c_{i+1}-c_i}{1-c_{i+1}}\min\left\{\norm*{V_{c_i}},\norm*{V_{c_{i+1}}}\right\},
\end{equation}
  On one hand,  \eqref{VsubEF1} implies
    \begin{equation}\label{VsubEF2}
    \begin{aligned}
        \norm*{V_\beta-V_\alpha}&\le\sum_{i=0}^j\frac{c_{i+1}-c_i}{(1-c_{i+1})(1-c_i)}(1-c_i)\norm*{V_{c_i}}\le\sum_{i=0}^j\frac{c_{i+1}-c_i}{(1-c_{i+1})(1-c_i)}(1-\alpha)\norm*{V_\alpha} \\
        &=(1-\alpha)\norm*{V_\alpha}\sum_{i=0}^j\left(\frac{1}{1-c_{i+1}}-\frac{1}{1-c_i}\right)=\frac{\beta-\alpha }{1-\beta
        }\norm*{V_\alpha},
    \end{aligned}
    \end{equation}
    where the second inequality is by  Theorem~\ref{Vine}(a), and the last equality follows from
    \[
    \sum_{i=0}^j\left(\frac{1}{1-c_{i+1}}-\frac{1}{1-c_i}\right)=\frac{1}{1-\beta}-\frac{1}{1-\alpha}=\frac{\beta-\alpha}{(1-\beta)(1-\alpha)}.
    \]
    On the other hand, by   \eqref{VsubEF1},
    \begin{equation}\label{VsubEF3}
        \norm*{V_{\beta}-V_{\alpha}}\le\sum_{i=0}^j\frac{c_{i+1}-c_i}{1-c_{i+1}}\norm*{V_{c_{i+1}}}\le \sum_{i=0}^j\frac{c_{i+1}-c_i}{1-\beta}\norm*{V_\beta}=\frac{\beta-\alpha}{1-\beta}\norm*{V_\beta},
    \end{equation}
where the second inequality is by Theorem~\ref{Vine}(b). Thus, \eqref{VsubEF2} and \eqref{VsubEF3} together imply (a).

Let us prove (b) similarly to the proof of (a). We observe that,  in view of Theorem~\ref{olip}(b) and Lemma~\ref{lempartition}(a), 
\begin{equation}\label{VsubEF4}
 \begin{aligned}
        \normo*{V_\beta-V_\alpha}\le\sum_{i=0}^j\normo*{V_{c_{i+1}}-V_{c_i}}&\le\sum_{i=0}^j\frac{\gamma_{\phi_i}(c_{i+1}-c_{i})}{1-c_{i+1}\gamma_{\phi_i}}\min\left\{\normo*{v_{c_i}^{\phi_i}},\normo*{v_{c_{i+1}}^{\phi_i}}\right\}\\ 
        &\le \sum_{i=0}^j\frac{\gamma (c_{i+1}-c_i)}{1-c_{i+1}\gamma}\min\left\{\normo*{V_{c_i}},\normo*{V_{c_{i+1}}}\right\}.
    \end{aligned}
   \end{equation} 
On one hand,  \eqref{VsubEF4} implies
    \begin{align}\label{VsubEF5}
        \normo*{V_\beta-V_\alpha}&\le\sum_{i=0}^j\frac{\gamma(c_{i+1}-c_i)}{(1-c_{i+1}\gamma)(1-c_i\gamma)}(1-c_i\gamma)\normo*{V_{c_i}}\le\sum_{i=0}^j\frac{\gamma(c_{i+1}-c_i)}{(1-c_{i+1}\gamma)(1-c_i\gamma)}(1-\alpha\gamma)\normo*{V_\alpha}\nonumber\\
        &=(1-\alpha\gamma)\normo*{V_\alpha}\sum_{i=0}^j\left(\frac{1}{1-c_{i+1}\gamma}-\frac{1}{1-c_i\gamma}\right)=\frac{\gamma(\beta-\alpha) }{1-\beta\gamma
        }\normo*{V_\alpha},
    \end{align}
     where the second inequality is by  Theorem~\ref{Vine}(a).   On the other hand, by \eqref{VsubEF4},
    \begin{equation}\label{VsubEF6}
        \normo*{V_{\beta}-V_{\alpha}}\le\sum_{i=0}^j\frac{\gamma(c_{i+1}-c_i)}{1-c_{i+1}\gamma}\normo*{V_{c_{i+1}}}\le \sum_{i=0}^j\frac{\gamma(c_{i+1}-c_i)}{1-\beta\gamma}\normo*{V_\beta}=\frac{\gamma(\beta-\alpha)}{1-\beta\gamma}\normo*{V_\beta},
    \end{equation}
where the second inequality is by Theorem~\ref{Vine}(b). Thus, \eqref{VsubEF5} and \eqref{VsubEF6} together imply (b). 
\end{proof}


\section{Conclusion}
In this paper, we proved various results regarding the differential properties of objective functions and value functions of finite MDPs in discount factors. The proofs rely heavily on the fact that, for finite MDPs, the objective functions and their derivatives can be written in matrix algebra. One prospective research direction is to establish sufficient conditions on when similar differential properties also hold for MDPs with infinite state and action sets.
The relevant known result is the lower semicontinuity of the value function $\alpha\mapsto V_\alpha(x),$ $x\in\X,$ for an MDP with a Borel state space $\X$ and with setwise continuous transition probabilities; \cite[Theorem 3.1(ii)]{FK}.

\section*{Acknowledgement}
This research was partially supported by the Office of Naval Research (ONR) under grant N000142412608.

\section*{Conflict of Interest}
The authors declare that they have no conflict of interest.

\appendix
\section{Proof of Lemma~\ref{lem1}}\label{ap1}
    If $k\le m-1,$ then both statements are obvious. The sufficiency for both statements is also obvious. 
    Let us assume $k\ge m$ and prove the necessity. 

    Note that $\mathbf{1}$ is an eigenvector of $P(\phi)$ corresponding to the eigenvalue $1.$ Thus, we can let $P(\phi)=QJQ^{-1}$ be its Jordan decomposition such that $Q\mathbf{e}_1=\mathbf{1},$ and $J$ is upper-triangular such that $J\mathbf{e}_1=\mathbf{e}_1,$ where $\mathbf{e}_1$ is the column vector whose first entry is $1$ and all other entries are $0.$ Then $P^n(\phi)=QJ^nQ^{-1}$ and $J^n\mathbf{e}_1=\mathbf{e}_1$ for all $n\in\mathbb{N}.$ In addition, since $J$ is a Jordan-block-diagonal matrix with at least one Jordan block being non-degenerate (for example, any block with diagonal entries being the eigenvalue $1$), all the degenerate Jordan blocks of $J$ must have sizes less than or equal $m-1.$ Hence, if $\mathbf{v}\in\R$ such that $J^k\mathbf{v}=\mathbf{0},$ then $J^{m-1}\mathbf{v}=\mathbf{0}.$ Let $\mathbf{u}=Q^{-1}r(\phi).$ 
    
    Let us prove  (a), where we only need to prove $P^{m-1}(\phi)r(\phi)=\mathbf{0}.$ Note that
    \[J^k\mathbf{u}=Q^{-1}P^k(\phi)Q\mathbf{u}=Q^{-1}P^k(\phi)r(\phi)=\mathbf{0}.\]
    Thus, $J^{m-1}\mathbf{u}=\mathbf{0}.$ 
    Therefore
    \[P^{m-1}(\phi)r(\phi)=QJ^{m-1}Q^{-1}Q\mathbf{u}=QJ^{m-1}\mathbf{u}=\mathbf{0}.\]
    
    Let us prove (b), where we only need to prove $P^{m-1}(\phi)r(\phi)=c\mathbf{1}.$ Note that
    \[J^k(c\mathbf{e}_1)=cJ^k\mathbf{e}_1=c\mathbf{e}_1=cQ^{-1}\mathbf{1}=Q^{-1}(c\mathbf{1})=Q^{-1}P^k(\phi)r(\phi)=Q^{-1}QJ^kQ^{-1}r(\phi)=J^k\mathbf{u},\] 
    which implies $\mathbf{u}=c\mathbf{e}_1+\tilde{\mathbf{u}}$ for some $\tilde{\mathbf{u}}$ such that $J^k\tilde{\mathbf{u}}=\mathbf{0}.$ Thus, $J^{m-1}\tilde{\mathbf{u}}=\mathbf{0}.$ 
    Therefore, 
    \begin{align*}
        P^{m-1}(\phi)r(\phi)=QJ^{m-1}Q^{-1}Q\mathbf{u}=QJ^{m-1}\mathbf{u}&=QJ^{m-1}(c\mathbf{e}_1+\tilde{\mathbf{u}})\\
        &=cQJ^{m-1}\mathbf{e}_1+QJ^{m-1}\tilde{\mathbf{u}}=cQ\mathbf{e_1}=c\mathbf{1}. \qedhere
    \end{align*}

\section{Proof of Lemma~\ref{lemdiff}}\label{ap2}
Let us prove (a). For simplicity, we only prove the case of $\frac{d}{d\theta^+}\norm{\mathbf{u}_\theta},$ as the case of $\frac{d}{d\theta^-}\norm{\mathbf{u}_\theta}$ is similar. We first claim that $\frac{d}{d\theta^+}\abs*{\mathbf{u}_\theta(x)}$ exists and $\abs*{\frac{d}{d\theta^+}\abs*{\mathbf{u}_\theta(x)}}=\abs*{\frac{d}{d\theta^+}\mathbf{u}_\theta(x)}$ for each $x\in\X.$ If $\mathbf{u}_\theta(x)=0,$ then
\[\lim_{h\to0^+}\frac{\abs*{\mathbf{u}_{\theta+h}(x)}-\abs*{\mathbf{u}_\theta(x)}}{h}=\abs*{\lim_{h\to0^+}\frac{\mathbf{u}_{\theta+h}(x)-\mathbf{u}_\theta(x)}{h}}=\abs*{\frac{d}{d\theta^+}\mathbf{u}_\theta(x)},\]  which implies that $\frac{d}{d\theta^+}\abs*{\mathbf{u}_\theta(x)}$ exists and $\frac{d}{d\theta^+}\abs*{\mathbf{u}_\theta(x)}=\abs*{\frac{d}{d\theta^+}\mathbf{u}_\theta(x)}.$ If $\mathbf{u}_\theta(x)\ne0,$ then the existence of $\frac{d}{d\theta}\mathbf{u}_\theta(x)$ implies continuity of $\mathbf{u}_\alpha$ at $\alpha=\theta.$ Thus, either $\abs*{\mathbf{u}_\alpha(x)}=\mathbf{u}_\alpha(x)$ or $\abs*{\mathbf{u}_\alpha(x)}=-\mathbf{u}_\alpha(x)$ for all $\alpha$ in a neighborhood of $\theta,$ which implies that $\frac{d}{d\theta^+}\abs*{\mathbf{u}_\theta(x)}$ exists, and either $\frac{d}{d\theta^+}\abs*{\mathbf{u}_\theta(x)}=\frac{d}{d\theta^+}\mathbf{u}_\theta(x)$ or $\frac{d}{d\theta^+}\abs*{\mathbf{u}_\theta(x)}=-\frac{d}{d\theta^+}\mathbf{u}_\theta(x).$ So, $\abs*{\frac{d}{d\theta^+}\abs*{\mathbf{u}_\theta(x)}}=\abs*{\frac{d}{d\theta^+}\mathbf{u}_\theta(x)}.$ Therefore, the claim is proved. Now let  
\[\hat{\X}=\{x\in\X:\abs{\mathbf{u}_\theta(x)}=\norm{\mathbf{u}_\theta}\}, \qquad M=\max_{x\in\hat{\X}}\frac{d}{d\theta}\abs{\mathbf{u}_\theta(x)},\qquad\hat{\X}_\uparrow=\left\{x\in\hat{\X}:\frac{d}{d\theta}\abs{\mathbf{u}_\theta(x)}=M\right\}.\]
By  continuity of $\abs{\mathbf{u}_\alpha}$ at $\alpha=\theta$ implied by the existence of $\frac{d}{d\theta}\mathbf{u}_\theta,$ the definitions of $\Hat{\X}$ and $\Hat{\X}_\uparrow$, and the definition of $M,$ there exists $\delta>0$ such that $\abs{\mathbf{u}_{\theta+h}(x)}>\abs{\mathbf{u}_{\theta+h}(y)}$ whenever $x\in\hat{\X}_\uparrow,$ $y\in\X\setminus\hat{\X}_\uparrow,$ and $h\in(0,\delta).$ 
Let $\epsilon>0.$ Since $\hat{\X}$ is finite, without loss of generality, let $\delta$ be small enough for
\[\abs*{\frac{\abs{\mathbf{u}_{\theta+h}(x)}-\abs{\mathbf{u}_\theta(x)}}{h}-\frac{d}{d\theta}\abs{\mathbf{u}_\theta(x)}}<\epsilon \qquad \forall x\in\hat{\X},\quad\forall h\in(0,\delta).\]
Then for each $h\in(0,\delta),$ there exists $x\in\hat{\X}_\uparrow,$ which may depend on $h,$ such that $\abs{\mathbf{u}_{\theta+h}(x)}=\norm{\mathbf{u}_{\theta+h}},$ and for the same $x\in\hat{\X}_\uparrow,$
\begin{equation*}
\abs*{\frac{\norm{\mathbf{u}_{\theta+h}}-\norm{\mathbf{u}_\theta}}{h}-M}=\abs*{\frac{\abs{\mathbf{u}_{\theta+h}(x)}-\abs{\mathbf{u}_\theta(x)}}{h}-\frac{d}{d\theta}\abs{\mathbf{u}_\theta(x)}}<\epsilon,
\end{equation*}
which implies $\frac{d}{d\theta^+}\norm{\mathbf{u}_\theta}$ exists and equals $M.$ Therefore, for all $x\in\hat{\X}_\uparrow,$
    \[\abs*{\frac{d}{d\theta^+}\norm*{\mathbf{u}_\theta}}=\abs{M}=\abs*{\frac{d}{d\theta^+}\abs{\mathbf{u}_\theta(x)}}=\abs*{\frac{d}{d\theta^+}\mathbf{u}_\theta(x)}\le\norm*{\frac{d}{d\theta^+}\mathbf{u}_\theta}=\norm*{\frac{d}{d\theta}\mathbf{u}_\theta}.\]

Let us prove (b). For simplicity, we only prove the case of $\frac{d}{d\theta^+}\norm{\mathbf{u}_\theta},$ as the case of $\frac{d}{d\theta^-}\norm{\mathbf{u}_\theta}$ is similar. Let $\overline{\X}=\{x\in\X:\mathbf{u}_\theta(x)=\max_{y\in\X}\mathbf{u}_\theta(y)\}$ and $\underline{\X}=\{x\in\X:\mathbf{u}_\theta(x)=\min_{y\in\X}\mathbf{u}_\theta(y)\}.$ Then $\normo{\mathbf{u}_\theta}=\mathbf{u}_\theta(x)-\mathbf{u}_\theta(y)$ for all $x\in\overline{\X}$ and all $y\in\underline{\X}.$ Moreover, let
\begin{align*}
    &\overline{M}=\max_{x\in\overline{\X}}\frac{d}{d\theta}\mathbf{u}_\theta(x),&\overline{\X}_\uparrow=\left\{x\in\overline{\X}:\frac{d}{d\theta}\mathbf{u}_\theta(x)=\overline{M}\right\},\\
    &\underline{M}=\min_{x\in\underline{\X}}\frac{d}{d\theta}\mathbf{u}_\theta(x),&\underline{\X}_\downarrow=\left\{x\in\underline{\X}:\frac{d}{d\theta}\mathbf{u}_\theta(x)=\underline{M}\right\}.
\end{align*}
Note that continuity of $\mathbf{u}_\alpha$ at $\alpha=\theta$ and the definitions of $\overline{M},\underline{M}$ imply the existence of $\delta>0$ such that $\mathbf{u}_{\theta+h}(x)>\mathbf{u}_{\theta+h}(y)$ whenever $x\in\overline{\X}_\uparrow,$ $y\in\X\setminus\overline{\X}_\uparrow,$ $h\in(0,\delta),$ and $\mathbf{u}_{\theta+h}(x)<\mathbf{u}_{\theta+h}(y)$ whenever $x\in\underline{\X}_\downarrow,$ $y\in\X\setminus\underline{\X}_\downarrow,$ $h\in(0,\delta).$ Let $\epsilon>0.$ Without loss of generality, let $\delta$ be small enough for
\[\abs*{\frac{\mathbf{u}_{\theta+h}(x)-\mathbf{u}_\theta(x)}{h}-\frac{d}{d\theta}\mathbf{u}_\theta(x)}<\frac{\epsilon}{2} \qquad \forall x\in\X,\quad \forall h\in(0,\delta).\]
Then for each $h\in(0,\delta),$ there exist $x\in\overline{\X}_\uparrow$ and $y\in\underline{\X}_\downarrow,$ which may depend on $h,$ such that $\mathbf{u}_{\theta+h}(x)=\max_{z\in\X}\mathbf{u}_{\theta+h}(z),$ $\mathbf{u}_{\theta+h}(y)=\min_{z\in\X}\mathbf{u}_{\theta+h}(z),$ and for the same pair of $x\in\overline{\X}_\uparrow,$ $y\in\underline{\X}_\downarrow,$
\begin{align*}
&\abs*{\frac{\normo{\mathbf{u}_{\theta+h}}-\normo{\mathbf{u}_\theta}}{h}-(\overline{M}-\underline{M})}=\abs*{\frac{\mathbf{u}_{\theta+h}(x)-\mathbf{u}_{\theta+h}(y)-\left[\mathbf{u}_\theta(x)-\mathbf{u}_\theta(y)\right]}{h}-\left[\frac{d}{d\theta}\mathbf{u}_\theta(x)-\frac{d}{d\theta}\mathbf{u}_\theta(y)\right]}\nonumber\\
&\le\abs*{\frac{\mathbf{u}_{\theta+h}(x)-\mathbf{u}_\theta(x)}{h}-\frac{d}{d\theta}\mathbf{u}_\theta(x)}+\abs*{\frac{\mathbf{u}_{\theta+h}(y)-\mathbf{u}_\theta(y)}{h}-\frac{d}{d\theta}\mathbf{u}_\theta(y)}<\epsilon,
\end{align*}
which implies that $\frac{d}{d\theta^+}\normo{\mathbf{u}_\theta}$ exists and is equal to $=\overline{M}-\underline{M}.$ 
Therefore, for all $x\in\overline{\X}_\uparrow$ and all $y\in\underline{\X}_\downarrow,$
\[\abs*{\frac{d}{d\theta^+}\normo{\mathbf{u}_\theta}}=\abs*{\overline{M}-\underline{M}}=\abs*{\frac{d}{d\theta}\mathbf{u}_\theta(x)-\frac{d}{d\theta}\mathbf{u}_\theta(y)}\le\normo*{\frac{d}{d\theta}\mathbf{u}_\theta}. \qedhere\]

\end{document}